\documentclass[12pt]{article}

\usepackage[utf8]{inputenc}
\usepackage[english]{babel}

\usepackage[T1]{fontenc}

\usepackage{verbatim}

\usepackage{hyperref}
\usepackage{amsmath,amsfonts,amssymb,amsthm,amscd}
\usepackage{nccmath}
\usepackage{mathrsfs,indentfirst,latexsym,float}
\usepackage{amscd}
\usepackage{eucal}
\usepackage{fancyhdr}
\usepackage[pdftex]{graphicx}
\numberwithin{equation}{section}

\newtheorem{theorem}{Theorem}[section]
\newtheorem{lemma}[theorem]{Lemma}
\newtheorem{proposition}[theorem]{Proposition}
\newtheorem{corollary}[theorem]{Corollary}
\newtheorem{question}[theorem]{Question}

\theoremstyle{definition}
\newtheorem{definition}[theorem]{Definition}

\newtheorem{remark}[theorem]{Remark}

\newcommand{\N}{\mathbb{N}} 
\newcommand{\Z}{\mathbb{Z}} 

\newcommand{\A}{\mathcal{A}}
\newcommand{\B}{\mathcal{B}}

\newcommand{\PS}{\mathcal{PS}}
\newcommand{\Synd}{\mathcal{S}}
\newcommand{\IP}{\mathcal{IP}}
\newcommand{\T}{\mathcal{T}}

\newcommand{\I}{\mathcal{I}}

\newcommand{\F}{\mathscr{F}}
\newcommand{\Ft}{\widetilde{\mathscr{F}}}

\newcommand{\Lin}{\mathcal{L}}

\newcommand{\abs}[1]{\left\lvert#1\right\rvert}

\newcommand{\norm}[1]{\lVert#1\rVert}

\begin{document}

\title{Strong transitivity properties for operators}
\author{J. B\`{e}s \footnote{Department of Mathematics and Statistics, Bowling Green State University, Bowling Green, OH 43403, USA. e-mail: jbes@math.bgsu.edu }, Q. Menet\footnote{Univ. Artois, EA 2462, Laboratoire de Mathématiques de Lens (LML), F-62300 Lens, France. e-mail: quentin.menet@univ-artois.fr},  A. Peris\footnote{Departament de Matem\`{a}tica Aplicada, IUMPA, Universitat Polit\`{e}cnica de Val\`{e}ncia, Edifici 7A, 46022 Val\`{e}ncia, Spain. e-mail: aperis@mat.upv.es} and Y. Puig\footnote{Department of Mathematics, Ben-Gurion University of the Negev, P.O.B. 653, 84105 Beer Sheva, Israel. e-mail: puigdedios@gmail.com}}

\date{}
\maketitle

\begin{abstract}
Given a Furstenberg family $\F$ of subsets of $\N$, an operator $T$ on a topological vector space $X$  is called $\F$-transitive provided for each non-empty open subsets $U$, $V$ of $X$ the set $\{n\in \Z_+ : T^n(U)\cap V\neq\emptyset\}$ belongs to $\F$. We classify the topologically transitive operators with a hierarchy of $\F$-transitive subclasses by considering families $\F$ that are determined by various notions of largeness and density in $\Z_+$.
  \end{abstract}

\section{Introduction}

Throughout this paper $X$ denotes a topological space and $\mathcal{U}(X)$ the set of non-empty open subsets of $X$.
When $X$ is a topological vector space, $\Lin (X)$ stands for the set of operators (i..e,  linear and continuous self-maps) on $X$. An operator $T\in \Lin (X)$ is called \emph{hypercyclic} if there exists a vector $x\in X$ such that
for each $V$ in $\mathcal{U}(X)$ the time return set
 \[
 N_T(x,V)=N(x, V):=\{n\geq 0: T^nx\in V\}
 \]
is non-empty, or equivalently (since $X$ has no isolated points) an infinite set. When $X$ is an $F$-space (that is, a complete and metrizable topological vector space), we know thanks to Birkhoff's transitivity theorem that $T$ is hypercyclic if and only if it is \emph{topologically transitive}, that is, provided
\[
N_T(U,V)=N(U, V):=\{n\geq 0: T^n(U)\cap V\neq \emptyset\}
\]
 is infinite for every $U, V\in \mathcal{U}(X)$.

Since 2004, several refined notions of hypercyclicity based on the properties of time return sets $N(x, V)$ have been investigated: frequent hypercyclicity~\cite{BG1,BG}, $\mathcal{U}$-frequent hypercyclicity~\cite{Sh1,BoGr16}, reiterative hypercyclicity \cite{BMPP}. More recently a general notion called $\mathcal{A}$-hypercyclicity, which generalizes the abovementioned  notions of hypercyclicity, has been used to investigate the different types of hypercyclic operators, see~\cite{BMPP,BoGr16}.

Our aim here is to investigate refined notions of topological transitivity based on
properties satisfied by the return sets $N(U,V)$. Some of these are already well-known, such as the topological notions of \emph{mixing}, \emph{weak-mixing}, and \emph{ergodicity}, say. Recall that a continuous self-map $T$ on $X$ is called \emph{mixing} provided $N(U, V)$ is cofinite for each $U, V\in\mathcal{U}(X)$. Also,  $T$ is called \emph{weakly mixing} whenever $T\times T$ is topologically transitive on $X\times X$, and this occurs precisely when the return set $N(U, V)$ is \emph{thick} (i.e. contains arbitrarily long intervals) for each $U, V\in\mathcal{U}(X)$   \cite{PS05}.
Finally,  $T$ is \emph{topologically ergodic}  provided  $N(U, V)$ is syndetic (i.e. has bounded gaps)
for each  $U, V\in \mathcal{U}(X)$. It is known that topologically ergodic operators are weakly mixing \cite{GrPe10}. The above mentioned notions may be stated through the concept of a (Furstenberg) family. The symbols $\mathbb{Z}$ and $\mathbb{Z}_+$ denote the sets of integers and of positive integers, respectively.

\begin{definition}
We say that a non-empty collection $\F$ of subsets of $\Z_+$ is a \emph{family}  provided that each set $A\in \F$ is infinite and that $\F$  is hereditarily upward (i.e. for any $A\in \F$, if $ B\supset A$ then $B\in \F$).
 The \emph{dual} family $\F^*$ of $\F$ is defined as the collection of subsets $A$ of $\Z_+$ such that  $A\cap B\neq\emptyset$ for every $B\in\F$.
\end{definition}

 Some standard families are the following: The family $\I$ of infinite sets, whose dual family $\I^*$ coincides with the family of cofinite sets. The family $\mathcal{T}$ of thick sets, whose dual family is $\mathcal{S}=\mathcal{T}^*$, the family of syndetic sets.
For a topologically transitive map $T$ a distinguished family is
 \[
 \mathcal{N}_T:=\{ A\subset \Z_+ : N_T(U,V)\subseteq A \mbox{ for some } U, V\in\mathcal{U}(X)\} .
 \]

From now on the symbol $\F$ will always denote a family.

\begin{definition}
\label{Fop}
 We say that a continuous map $T$ on $X$ is \emph{$\F$-transitive} (or an $\F$-map, for short) provided
 $\mathcal{N}_T\subset \F$, that is, provided $N(U, V)\in \F$ for each $U, V\in \mathcal{U}(X)$.
If in addition $X$ is a topological vector space and $T\in \Lin (X)$ we call $T$ an \emph{$\F$-transitive operator} (or $\F$-operator for short).
\end{definition}

Hence  the  $\I$-operators are precisely those operators which are  topologically transitive, and
the $\I^*$-operators and $\T$-operators are precisely those which are mixing and weak mixing, respectively. The $\mathcal{T}^*=\mathcal{S}$-operators, that is, the  {topologically ergodic} operators.

We present here some new classes of topologically transitive operators by considering families $\F$ defined in terms of various notions of density and largeness in $\Z_+$. A  hierarchy of fourteen  classes (which include the earlier mentioned classes defined by properties of return sets $N(x, V)$) appears  in Figure~\ref{fig:fig2} and summarizes our findings. We stress that while trivially any $\F_1$-map is an $\F_2$-map when $\F_1\subset \F_2$,  it is possible that the classes of $\F_1$-operators and of $\F_2$-operators coincide even if $\F_1$ is strictly contained in $\F_2$ (see e.g., Proposition~\ref{synd.op.is.hered.synd}).

The paper is organized as follows. In Section~\ref{ftrans} we describe some general facts about families  $\F$ and their corresponding $\F$-transitive maps and operators. In Theorem~\ref{prop.Fcrit} we provide an extension of the Hypercyclicity Criterion that ensures an operator to be $\F$-transitive. We apply this criterion in Section~\ref{fshifts} to characterize $\F$-transitivity among unilateral and bilateral weighted backward shift operators on $c_0$ and $\ell_p$ $(1\le p<\infty)$ spaces. To illustrate, we establish in Corollary~\ref{teuws} that a unilateral backward shift $B_w$ is topologically ergodic precisely when its weight sequence $w=(w_n)_n$ satisfies that each set
\[
A_M=\{ n:   |\prod_{j=1}^n w_j |> M \}        \quad (M>0)
\]
is syndetic.
 Section~\ref{return} is dedicated to $\F$-operators induced by families $\F$ given by sets of positive or full  (lower or upper) asymptotic density or Banach density.
 In Section~\ref{filt}, we look at $\F$-operators induced by families $\F$ commonly used in Ramsey theory, and we compare the classes that we obtain with the class of reiteratively hypercyclic operators (Subsection~\ref{ahyp}). Some natural questions conclude the paper.


 \section{$\F$-Transitivity}\label{ftrans}

In this section we introduce a sufficient condition for an operator to be an $\F$-operator, which we call the $\F$-Transitivity Criterion, and it is in the same vein of the Hypercyclicity Criterion. Moreover, we will study the notion of hereditarily $\F$-operator.

We will be interested in the following  three special properties a family $\F$ can have: being a \emph{filter}, being \emph{partition-regular}, and being \emph{shift-invariant}. We use the following notation: given two families $\F_1$ and $\F_2$
 \[
 \F_1\cdot \F_2:=\{ A\cap B : A\in\F_1, \ \ B\in \F_2\}.
 \]
Obviously,  $\F_1\subset \F_1\cdot \F_2$ and
$\F_2\subset \F_1\cdot \F_2$.
A family $\F$ is a \emph{filter} provided it is invariant under finite intersections (i.e., provided $\F \cdot \F\subset \F$). Say, the family $\I^*$  of cofinite sets is a  filter while  the families $\I$ and $\mathcal{S}$ of infinite sets and of syndetic sets are not.

The second property, that of being partition regular, will be useful for us
to identify filters. A family $\F$ on $\Z_+$ is said to be \emph{partition regular} if for every $A\in\mathscr{F}$ and any finite partition $\{A_1, \dots, A_n\}$ of $A$, there exists some $i=1,\dots ,n$ such that $A_i\in \mathscr{F}$. The family $\I$ is an example of partition regular family, while the families $\I^*$, $\mathcal{T}$ and $\mathcal{S}$ are not. Later we will see other examples of partition regular families: the family of piecewise syndetic sets (see Remark~\ref{remfcrit}), the family of sets with positive upper (Banach) density (see Section~\ref{return}), the families of $\Delta$-sets and of $\IP$-sets (see Section~\ref{filt}).

   \begin{lemma}
   \label{part.reg}
   Given a family  $\F$, the following are equivalent:
   \begin{enumerate}
    \item[(I)] $\F$ is partition regular,
    \item[(II)] $A\cap A'\in \F$ for every $A\in \F$ and $A'\in \F^*$ (i.e., $\F\cdot \F^*\subset \F$),
    \item[(III)] $\F^*$ is a filter.
    \end{enumerate}
   \end{lemma}

   \begin{proof}
   (I) $\Longrightarrow$ (II): Given $A\in \F$ and $A'\in \F^*$ it is clear that $A\cap A'\neq \emptyset$ by definition of dual family. Since $(A\cap A')\cup (A\setminus A')=A$, either $A\cap A'\in \F$ or $A\setminus A'\in\F$ by (I). Since $(A\setminus A')\cap A'=\emptyset$, by definition of dual family we necessarily have $A\cap A'\in\F$.

    (II) $\Longrightarrow$ (III): For arbitrary $A',B'\in \F^*$ and $A\in \F$, by applying (II) and the definition of dual family we have $A\cap (A'\cap B')=(A\cap A')\cap B'\neq \emptyset$, which yields that $\F^*$ is a filter.

     (III) $\Longrightarrow$ (I): We will just show that, given $A\in \F$ and $A_1,A_2\subset \Z_+$ such that $A_1\cup A_2=A$ and $A_1\cap A_2=\emptyset$, then either $A_1\in \F$ or $A_2\in \F$. The general case can be deduced by an inductive process. Since $\F=\F^{**}$, we need to show that $A_i\cap A'\neq \emptyset$ for every $A'\in\F^*$, for $i=1$ or $i=2$. Suppose that there exist $A',B'\in \F^*$ with $A_1\cap A'=\emptyset$ and $A_2\cap B'=\emptyset$. Since $\F^*$ is a filter, then $C':=A'\cap B'\in \F^*$. Thus,
     \[
     \emptyset \neq A\cap C'\subset (A_1\cap A')\cup (A_2\cap B')=\emptyset ,
     \]
     which is a contradiction.
  \end{proof}

Notice that $(\F^*)^{*}=\F$ for any family $\F$: the inclusion $\F\subset (\F^*)^*$ is immediate. Conversely, if $A\in (\F^*)^*$, then $\Z_+\setminus A\not \in \F^*$ by the definition of a dual family. This means that there exists $B\in\F$ such that $B\cap (\Z_+\setminus A)=\emptyset$. That is, $B\subset A$, which gives $A\in \F$.

Thus any family is a dual family, and Lemma~\ref{part.reg} also gives that a family $\F$ is a filter if and only if $\F^* \cdot \F \subset \F^*$ and if and only if $\F^*$ is partition regular. Another consequence of Lemma~\ref{part.reg} is that any family $\F$ that is both a filter and partition regular (called an \emph{ultrafilter}) must satisfy $\F=\F^*$.

Finally, our third property: A family $\F$ on $\Z_+$ is said to be \emph{shift$_-$-invariant} provided for every $i\in \Z_+$ and  each $A\in \mathscr{F}$, we have $(A-i)\cap \Z_+\in \F$. We say that $\F$ is called  \emph{shift$_+$-invariant} if for every $i\in \Z_+$ and  each $A\in \mathscr{F}$, we have $A+i\in \F$. When $\F$ is both, shift$_-$-invariant and shift$_+$-invariant, we simply call it \emph{shift invariant}.
For instance, the families of infinite sets, cofinite sets, thick sets and syndetic sets are shift invariant.

We may gain shift invariance by reducing a family. Given a family $\F$, we define
\[
\Ft_+=\{ A\subset \Z_+ \ : \ \forall N\in\Z_+ \ \exists B\in \F  \mbox{ such that }  A\supset B + [0,N] \},
\]
\[
\Ft_-=\{ A\subset \Z_+ \ : \ \forall N\in\Z_+ \ \exists B\in \F  \mbox{ such that }  A\supset (B + [-N,0])\cap\Z_+ \},
\]
\[
\Ft=\{ A\subset \Z_+ \ : \ \forall N\in\Z_+ \ \exists B\in \F  \mbox{ such that }  A\supset (B + [-N,N])\cap\Z_+ \}.
\]
So for any family $\F$ we have the inclusions $\Ft\subset \Ft_+\subset \F$ and $\Ft\subset \Ft_-\subset \F$, and that $\Ft_-$ is shift$_+$-invariant, $\Ft_+$ is shift$_-$-invariant, and $\Ft$ is shift invariant.

\begin{lemma}\label{techft}
If $\F$ is a filter on $\Z_+$, so is $\Ft$.  Moreover, for any family $\F$ satisfying  $\Ft \cdot \Ft\subset \F$  the subfamily $\Ft$ is a filter.
\end{lemma}

\begin{proof}
Let $A_1,A_2\in \Ft$. We have to show that $A_1\cap A_2\in \Ft$. Given $N\in \N$, there are $B_1(N),B_2(N)\in \F$ such that $(B_1(N)+[-2N,2N])\cap\Z_+\subset A_1$ and
$(B_2(N)+[-2N,2N])\cap\Z_+\subset A_2$. For $i=1,2$ we define
\[
\bar{A}_i(N):=\bigcup_{J\geq N} (B_i(J)+[-J,J])\cap \Z_+.
\]
Clearly $\bar{A}_1(N),\bar{A}_2(N)\in\Ft$ for each $N\in\N$. By hypothesis, $B(N):=\bar{A}_1(N) \cap \bar{A}_2(N)\in\F$, $N\in\N$. To prove that $A_1\cap A_2\in\Ft$ we just need to
show that $(B(N)+[-N,N])\cap\Z_+\subset A_1\cap A_2$ for every $N\in\N$. Indeed, given $N\in\N$ and $m\in (B(N)+[-N,N])\cap\Z_+$, we write $m=k(N)+l(N)$ with $k(N)\in B(N)$ and $l(N)\in [-N,N]$.
By definition of $B(N)$ we have
\[
k(N)=k_1(J_1)+l_1(J_1)=k_2(J_2)+l_2(J_2)
\]
\[
\mbox{for some } k_i(J_i)\in B_i(J_i), \ \ l_i(J_i)\in [-J_i,J_i], \  J_i\geq N, \ i=1,2.
\]
Thus
\[
m=k_1(J_1)+l_1(J_1)+l(N)\in (B_1(J_1)+[-2J_1,2J_1])\cap\Z_+\subset A_1,
\]
and, analogously, $m\in A_2$, which yields the result.
\end{proof}


The rest of the section is dedicated to $\F$-maps and $\F$-operators.
Every $\Ft$-map is an $\F$-map, since $\Ft \subset \F$. The next lemma gives conditions for the converse, and is used in Proposition~\ref{bilateral.weighted.shifts.F.oper}.

\begin{lemma}\label{lthick}
Let $\F$ be a family on $\Z_+$ and let $T$ be a $\F$-map. The following are equivalent.
\begin{enumerate}
\item[(i)] $T$ is weakly mixing,
\item[(ii)] $T$ is an $\Ft$-map.
\end{enumerate}
\end{lemma}

\begin{proof}
(i) implies (ii): Given $N\in \N$ and $U, V\in \mathcal{U}(X)$, since $T$ is weakly mixing, by Furstenberg result we know that $\mathcal{N}_T$ is a filter, so there are $U', V'\in \mathcal{U}(X)$ such that
\[
N(U',V')\subset N(T^{-m}(U),V)\cap N(U,T^{-m}(V)),
\]
for $m=0,\dots ,N$. By $\F$-transitivity we have $ N(U',V')\in \F$. We then conclude that
$(N(U',V')+[-N,N])\cap \Z_+\subset N(U,V)$, and $T$ is $\Ft$-transitive.

(ii) implies (i): If $T$ is an $\Ft$-map, since every element of $\Ft$ is thick, we have that $\mathcal{N}_T$ consists of thick sets and,
as we already recalled in the introduction, this means that $T$ is weakly mixing.
\end{proof}

To state the $\F$-Transitivity Criterion, we recall the notion of \emph{limit along a family $\F$}: Given a sequence $\{x_n\}_n$ in $X$ and
 $x\in X$, we say that
\[
\mbox{
 $\mathscr{F}-\lim_n x_n=x$, or that $x_n \overset{\F}{\to} x$,}
\]
 provided
$
\{ n\in \Z_+ : x_n\in U\}  \in \F$ for each neighbourhood $U$ of $x$.

 \begin{theorem} {\rm ($\F$-Transitivity Criterion)} 
 \label{prop.Fcrit} Let $T$ be an operator on a topological vector space $X$ and let  $\mathscr{F}$ be a family on $\Z_+$ such that $\Ft$ is a filter. Suppose there exist $D_1, D_2$ dense sets in $X$, and (possibly discontinuous) mappings $S_n: D_2\rightarrow X$, $n\in\N$ satisfying
\begin{itemize}
\item[(a)]  $\F$-$\lim_nT^n(x)=0$ for every $x\in D_1$
\item[(b)]  $\F$-$\lim_n (S_n(y),T^nS_n(y))=(0,y)$ for every $y\in D_2$.
\end{itemize}
 Then $T$ is an $\Ft$-operator.
 \end{theorem}
 \begin{proof}

 Let $U, V\in \mathcal{U}(X)$. We fix $U', V'\in \mathcal{U}(X)$ and a $0$-neighbourhood $W$ such that $U'+W\subset U$ and $V'+W\subset V$. Given $N\in\N$, pick $x\in D_1\cap T^{-N}U'$ and $y\in D_2\cap T^{-N}V'$. By continuity of $T$ we easily get
 \[
 \Ft_+-\lim_nT^n x=0 ,
 \]
 which yields $N(T^{-N}U',W)\in \Ft_+$. That is, there is $A\in \F$ such that $A+[0,2N]\subset N(T^{-N}U',W)$. Therefore,
 \[
 (A+[-N,N])\cap \Z_+\subset (N(T^{-N}U',W)-N)\cap\Z_+\subset N(U',W),
 \]
 and, since $N$ was arbitrary, we have that $N(U',W)\in \Ft$.

 Also,  we find a $0$-neighbourhood $W'\subset W$ with $T^m(W')\subset W$ and $y+W'\subset T^{-N}V'$, $m=0,\dots ,2N$. There is $A\in \F$ such that $S_ny\in W'$ and $T^nS_n(y)\in y+W'$ for all $n\in A$. Thus,
\[
\left(T^{(n-m)}(T^mS_n(y)),T^mS_n(y)\right)\in (y+W',T^m(W'))\subset (T^{-N}V',W) ,
\]
for $m=0,\dots ,2N$ and for every $n\in A$. In particular, $(A+[-N,N])\cap\Z_+ \subset N(W,V')$. Since $N$ was arbitrary, we obtain that $N(W,V')\in\Ft$. Therefore,
\[
N(U,V)\supset N(U'+W,V'+W)\supset N(U',W)\cap N(W,V')\in \Ft \cdot \Ft  \subset \Ft,
\]
that is, $T$ is an $\Ft$-operator.
 \end{proof}

\begin{remark}\label{remfcrit}
\begin{enumerate}
\item \
By Lemma~\ref{techft} the assumption that $\Ft$ be a filter is trivially satisfied in the case that $\F$ is a filter, but Theorem~\ref{prop.Fcrit} applies beyond this case. For instance, the family $\F=\mathcal{S}$  of syndetic sets is not a filter, and $\widetilde{\mathcal{S}}= \mathcal{TS}$ is the family of \emph{thickly syndetic sets}, which is a filter. So every operator that satisfies the $\mathcal{S}$-Transitivity Criterion is a $\mathcal{TS}$-operator.

In contrast, if we consider the family of \emph{piecewise syndetic sets} $\mathcal{PS}=\mathcal{TS}^*=\mathcal{T}\cdot \mathcal{S}$ (i.e., $A$ is piecewise syndetic if, and only if, it is the intersection of a thick set with a syndetic set), then $\widetilde{\mathcal{PS}}=\mathcal{T}$, and $\emptyset\in \mathcal{T}\cdot \mathcal{T}$. Thus the hypotheses of Theorem~\ref{prop.Fcrit} are not satisfied. Actually, it is not hard to construct an operator $T$ such that conditions (a) and (b) in Theorem~\ref{prop.Fcrit} are satisfied for $\F=\mathcal{PS}$, with $T$ not even transitive.

\item \ Another remarkable case is provided by, given a strictly increasing sequence $(n_k)_k$ in $\N$, considering the filter
\[
\F:=\{ A\subset \N : \exists j\in\N \mbox{ with } A\supset \{ n_k: k\geq j\}\}.
\]
In this case Theorem~\ref{prop.Fcrit} turns out to coincide with the classical Hypercyclicity Criterion. Moreover, since the Hypercyclicity Criterion characterizes the weakly mixing operators on separable $F$-spaces \cite{BePe}, we have that every weakly mixing operator $T$ on a separable $F$-space $X$ supports a strictly increasing sequence $(n_k)_k$ in $\N$ such that $T$ is an $\F$-operator, where
\[
\F:=\{ A\subset \N : \forall N\in\N \ \exists j\in\N \mbox{ with } A\supset \{ n_k: k\geq j\}+[-N,N]\}.
\]
\item \ We note that for an $\Ft$-operator $T$ with $\Ft$ a filter it is not true in general that $T$ must satisfy the $\mathcal{G}$-Transitivity Criterion for some filter $\mathcal{G}\subset \Ft$: just consider the family $\F=\I^*$ of cofinite sets and the fact that there exist mixing operators not satisfying Kitai's Criterion \cite[Theorem 2.5]{Gri}.

\item \ Recall that for the case  $\F=\I$, Furstenberg \cite[Proposition II.3]{Furstenberg67} showed that once $T\oplus T$ is an $\I$-map on $X^2$, every direct sum $\oplus_{j=1}^r T$  on $X^r$ is an $\I$-map too ($r\in \N$).
 The assumptions of the $\F$-Transitivity Criterion on an operator $T$ clearly ensure that (any direct sum $\oplus_{j=1}^r T$ will satisfy the $\F$-Transitivity Criterion on the space $X^r$ and thus that) $\oplus_{j=1}^r T$ is an $\Ft$-operator on $X^r$, for every $r\in\N$.
\end{enumerate}
\end{remark}

We next introduce the concept of a hereditarily $\F$-operator, and we establish links with that of an $\F$-operator.

\begin{definition}
  We say that a continuous map $T$  is a \emph{hereditarily $\mathscr{F}$-map} if  $N(U, V)\cap A\in\mathscr{F}$ for every $U, V\in \mathcal{U}(X)$ and every $A\in \mathscr{F}$ (that is,
  $\mathcal{N}_T\cdot \F\subset \F$). In addition, if $X$ is a topological vector space and $T\in \Lin (X)$, we say that $T$ is a \emph{hereditarily $\mathscr{F}$-operator}.
 \end{definition}
Clearly, hereditarily $\F$-maps are $\F$-maps. Moreover, they are automatically $\F^*$-maps since $\mathcal{N}_T\cdot \F\subset \F \not \ni \emptyset$. Also, for a filter $\F$ the concepts of $\F$-map and hereditarily $\F$-map are equivalent. More generally, we have:

   \begin{proposition}
   \label{Fstar.maps}
Let $T$ be a continuous map on a complete separable metric space $X$ without isolated points.
\begin{enumerate}
\item[(A)] Let $\F$ be a  partition regular family. Then the following are equivalent:
\begin{enumerate}
\item[(1)] $T$ is an $\F^*$-map;

\item[(2)] $T$ is a hereditarily $\F^*$-map;

\item[(3)] $T$ is a hereditarily $\F$-map;

\item[(4)] $hcA:=\{ x\in X : \overline{\{T^nx:n\in A\}}=X\}$ is a dense ($G_\delta$) set in $X$ for any $A\in \F$.
\end{enumerate}
\item[(B)] Let $\F$ be a  filter. Then the following are equivalent:
\begin{enumerate}
\item[(i)] $T$ is an $\F$-map;

\item[(ii)] $T$ is a hereditarily $\F$-map;

\item[(iii)] $T$ is a hereditarily $\F^*$-map;

\item[(iv)] $hcA:=\{ x\in X : \overline{\{T^nx:n\in A\}}=X\}$ is a dense ($G_\delta$) set in $X$ for any $A\in \F^*$.
\end{enumerate}
\end{enumerate}
\end{proposition}

\begin{proof} We will just show (A) since (B) follows by taking duals and Lemma~\ref{part.reg}. Indeed, condition (1) is equivalent to (2) because $\F^*$ is a filter. The fact that (1) implies (3) is a consequence of Lemma~\ref{part.reg} too, while the converse was already noticed before for general families. Finally  the equivalence between (1) and (4) can be shown in a similar way as Birkhoff's transitivity theorem \cite{GrPe}.
\end{proof}

  Note that when considering   the family $\F=\I$ of infinite sets in Proposition~\ref{Fstar.maps} (A) we obtain the known equivalences for mixing maps.

\begin{remark}
By the same argument for an operator $T$ on a separable topological vector space $X$, the first three equivalences of statements $(A)$ and $(B)$ still hold.  We also point out that as with the hypercyclic case we have the following comparison principle for $\F$-maps and transference principle for  $\F$-operators, see \cite[Chapter 12]{GrPe}.
\begin{enumerate}
\item[1.]\  ($\F$-Comparison Principle) Any quasifactor of an $\F$-map is an $\F$-map. Indeed, let $T:X\to X$ be an $\F$-map and let
$S:Y\to Y$ and $\phi:X\to Y$ be maps so that $\phi\circ T = S\circ \phi$, where $\phi$ has dense range. Then for any non-empty open subsets $U$ and $V$ of $Y$ we have
$
N_S(U, V)= N_T(\phi^{-1}(U), \phi^{-1}(V) )\in \F
$.

\item[2.]\ (Transference Principle) Let $\F$ be a family and let $T$ be an operator on a topological vector space $X$ so that each operator $S$ on an $F$-space  that is quasi-conjugate to $T$ via an operator (that is,  it supports a dense range operator $J:X\to Y$ with $J T= S J$)  is an $\F$-map.  Then $T$ is an $\F$-map.

\end{enumerate}

\end{remark}

 \section{ $\F$-transitive weighted shift operators}\label{fshifts}


Each bounded bilateral weight sequence $w=(w_k)_{k\in \Z}$, induces a \emph{bilateral weighted backward shift} operator $B_w$ on $X=c_0(\Z)$ or $\ell^p(\Z)$ ($1\leq p<\infty$) given by $B_{w}e_k:=w_{k}e_{k-1}$, where $(e_k)_{k\in \Z}$ denotes the canonical basis of $X$.

Similarly,  each bounded sequence $w=(w_n)_{n\in \N}$ induces a \emph{unilateral weighted backward shift} operator $B_w$ on $X=c_0(\Z_+)$ or $\ell^p(\Z_+)$ $(1\leq p<\infty)$, given by $B_{w}e_n:=w_{n}e_{n-1}, n\geq 1$ and $B_{w}e_0:=0$, where $(e_n)_{n\in \Z_+}$ denotes the canonical basis of $X$.

Our characterization of $\F$-transitive weighted backward shifts will rely on the properties of the sets $A_{M,j}$ and $\bar{A}_{M,j}$ defined as
  \[
\begin{aligned}
  A_{M,j}&:= \Big\{n\in \N : \prod_{i=j+1}^{j+n} \abs{w_i}>M\Big\}\\
   \bar{A}_{M,j}&:=\Big\{n\in \N : \frac{1}{\prod_{i=j-n+1}^j\abs{w_i}}>M\Big\},
\end{aligned}
  \]
  where $M>0$ and $j\in \Z$. In the case $j=0$, we just write $A_M, \bar{A}_M$ instead of $A_{M, 0}, \bar{A}_{M, 0}$ respectively. We note that Salas' \cite{Sal}  characterization of hypercyclic (i.e., transitive) bilateral weighted shifts on the above sequence spaces may be formulated as
  \[
  B_w \mbox{ is hypercyclic } \Leftrightarrow \ \forall M>0 \ \ \forall N\in\N \ \ \bigcap_{j=-N}^N (A_{M,j}\cap \bar{A}_{M,j})\neq \emptyset.
  \]
  In other words, since $A_{M',j}\subset A_{M,j}$ and $\bar{A}_{M',j}\subset \bar{A}_{M,j}$ whenever $M'>M>0$, the collection of subsets $\{A_{M,j},\bar{A}_{M,j}\}_{M>0,j\in\Z}$ should form a filter subbase for the hypercyclicity of $B_w$. In that case, we denote by $\mathcal{A}_w$ the generated filter. Therefore, for the characterization of weighted shifts $B_w$ that are $\F$-operators for a certain family $\F$ we need to assume that $\mathcal{A}_w$ is a filter.

  When $B_w$ is hypercyclic (i.e., when $\mathcal{A}_w$ is a filter), we can describe a filter base of $\mathcal{A}_w$, which will be very useful in the characterization of weighted shifts that are $\F$-operators, and it is given by the collection of sets
  \[
  \{ A_{M,j}\cap \bar{A}_{M,j} \ : \ M>0 \mbox{ and } j\in \N\}.
  \]
  Actually, this is a consequence of the observation that, if  $M_1,M_2>0$ and  $j_1,j_2,j_3\in\Z$ with $j_3>\max\{\abs{j_1},\abs{j_2}\}$, then there is $M_3>0$ such that
  \[
  A_{M_3,j_3}\subset A_{M_1,j_1} \cap A_{M_2,j_2} \ \ \mbox{ and } \ \ \bar{A}_{M_3,j_3}\subset \bar{A}_{M_1,j_1} \cap \bar{A}_{M_2,j_2}.
  \]
  Indeed, let $M:=\sup_{i\in\Z}\abs{w_i}$. We fix $M_3>K(M_1+M_2)(1+M)^{2j_3}$, where
  \[
  K:=1+\max_{-j_3\leq m_1\leq m_2\leq j_3} \prod_{i=m_1}^{m_2}\abs{w_i}^{-1}.
  \]
  If $n\in A_{M_3,j_3}$ then
  \[
  \prod_{i=j_1+1}^{j_1+n}\abs{w_i} =\left(\prod_{i=j_3+1}^{j_3+n}\abs{w_i}\right) \frac{\prod_{i=j_1+1}^{j_3}\abs{w_i}}{\prod_{i=j_1+n+1}^{j_3+n}\abs{w_i}}>M_3 \frac{\prod_{i=j_1+1}^{j_3}\abs{w_i}}{M^{j_3-j_1-1}}
  >M_1.
  \]
  That is, $n\in A_{M_1,j_1}$. The same argument shows $n\in A_{M_2,j_2}$. Analogously, we also have $\bar{A}_{M_3,j_3}\subset \bar{A}_{M_1,j_1} \cap \bar{A}_{M_2,j_2}$.

  \begin{proposition}
  \label{bilateral.weighted.shifts.F.oper}
  Let $B_w$ be a  bilateral weighted backward shift on $X=c_0(\Z)$ or $\ell^p(\Z)$, $1\leq p< \infty$.  Then the following are equivalent:
\begin{enumerate}
\item[(1)] $B_w$ is an $\Ft$-operator;
\item[(2)] $B_w$ is an $\F$-operator;
\item[(3)] for every $j\in \N$ and $M>0$, $A_{M, j}\cap \bar{A}_{M, j}\in  \F$;
\item[(4)] $B_w$ is hypercyclic, $\mathcal{A}_w\subset \F$, and $B_w$ satisfies the $\mathcal{A}_w$-Criterion.
\end{enumerate}
In addition, if $\Ft$ is a filter, then the above conditions are equivalent to
\begin{enumerate}
\item[(5)] for every $j\in \N$ and $M>0$ we have $A_{M, j}\in \F$ and $\bar{A}_{M, j}\in \F$.
\end{enumerate}
\end{proposition}

 \begin{proof}
Obviously, (1) implies (2). The reverse implication is a consequence of Lemma~\ref{lthick} since transitive weighted shifts are weakly mixing. Also, (4) implies (2).
To show that (2) implies (3), given $N,j\in\N$ arbitrary, we must find nonempty open sets $U,V\subset X$ such that
\begin{equation}\label{incweight}
N(U,V) \subset A_{N,j}\cap \bar{A}_{N,j}.
\end{equation}
Indeed, we fix $R>N$,
\[
 U:=\{x\in X \ : \ \abs{x_j}>\frac{1}{R}\} \cap \{x\in X \ : \ \norm{x}<1\},
 \]
and we set
 \[
V=\big\{x\in X: \big\|x- (N+1)e_j\big\|<\frac{1}{R^2}\big\}.
 \]
If $m\in N(U, V)$ and $x\in U$ is such that $B_w^mx\in V$, then
 \begin{equation}
 \label{cond.concerning.j.in.F}
 \abs{\left(\prod_{i=j+1}^{j+m}w_i \right)x_{j+m}-(N+1)}<\frac{1}{R^2}<1 \ \   \mbox{ and }
 \end{equation}
 \begin{equation}
 \label{cond.concerning.j.not.in.F}
\abs{\left(\prod_{i=l+1}^{l+m}w_i \right)x_{l+m}}<\frac{1}{R^2} \ \ \mbox{ if } l\neq j.
 \end{equation}
 Since $x\in U$, we deduce from \eqref{cond.concerning.j.in.F} that
 \[
 \prod_{i=j+1}^{j+m}\abs{w_i}>\left(\prod_{i=j+1}^{j+m}\abs{w_i}\right)\abs{x_{j+m}}> N  ,
 \]
 which implies that $m\in A_{N, j}$.

 On the other hand,  $B^m_wx\in V$ forces $m>0$ since $U$ and $V$ do not intersect.   Thus, $l:=j-m\neq j$, and
  \eqref{cond.concerning.j.not.in.F} implies
 \[
 \left(\prod_{i=j-m+1}^j\abs{w_i}\right) < \left(\prod_{i=j-m+1}^j\abs{w_i}\right)R \abs{x_j}<\frac{1}{R}<\frac{1}{N},
 \]
that yields $m\in\bar{A}_{N,j}$. Thus the inclusion \eqref{incweight} is satisfied, and property (3) holds.

   To prove that (3) implies (4), since $B_w$ is hypercyclic (i.e., $\mathcal{A}_w$ is a filter) and $\mathcal{A}_w\subset \F$
   because $\F$ contains a basis of $\mathcal{A}_w$, we just need to show that $B_w$ satisfies the $\mathcal{A}_w$-criterion.

Let $D$ be the set of all finitely supported vectors in $X$ and let $S_w$ be the weighted forward shift defined on $D$ by
  \[
  S_w e_i:=\frac{1}{w_{i+1}} e_{i+1}.
  \]
If we consider $S_n:=S^n_w$ then we have $B_w^nS_nx=x$ for every $x\in D$. It suffices to show that
\begin{itemize}
\item $\mathcal{A}_w$-$\lim_nB_w^n x=0$ for every $x\in D$;
\item $\mathcal{A}_w$-$\lim_n S_n x=0$ for every $x\in D$.
\end{itemize}

For the rest of the proof we assume that $X=\ell^p(\mathbb{Z})$ with $1\le p<\infty$. The proof is similar if $X=c_0(\Z)$.
Let $x\in D$, $\varepsilon>0$ and $V_\varepsilon:=\{x\in \ell^p(\Z): \norm{x}<\varepsilon\}$.
First, we show that $ \{n\in\mathbb{N}: B_w^n x\in V_{\varepsilon}\}\in\mathcal{A}_w$. Since $x\in D$, we can write $x=\sum_{j=-m}^mx_je_j$ for some $m\in\N$ and we then have
   \[
   B_w^nx= \sum_{j=-m-n}^{m-n}  \left(\prod_{i=j+1}^{j+n} w_i\right)x_{j+n} e_{j}.
 \]
  Let $M=\|x\|_{\infty} 2m/\varepsilon$ and $n\in \bigcap_{j=-m}^m \bar{A}_{M, j}\in \mathcal{A}_w$. We have
   \[
\norm{B_w^nx}^p=\sum_{j=-m}^m \abs{\prod_{i=j-n+1}^j w_i}^p\abs{x_j}^p<\sum_{j=-m}^m  \left(\frac{\varepsilon}{\|x\|_{\infty} 2m}\right)^p\abs{x_j}^p<\varepsilon ^p,
   \]
   which implies
   \[
   \bigcap_{j=-m}^m \bar{A}_{M,j}\subseteq \{n\in \N: B_w^ny\in V_\varepsilon\},
   \]
thus $ \{n\in \N: B_w^ny\in V_\varepsilon\} \in \mathcal{A}_w$. It remains to show that $ \{n\in\N: S_n x\in V_{\varepsilon}\}\in\mathcal{A}_w$. Indeed, we have
   \[
   S_n x=S^n_w x= \sum_{j=-m}^m \frac{x_{j}}{\prod_{i=j+1}^{j+n} w_i} e_{j+n}.
   \]
  Let $M=\|x\|_{\infty}2m /\varepsilon$ and $n\in \bigcap_{j=-m}^m A_{M, j}$. We then have
   \[
\norm{S_nx}^p=\sum_{j=-m}^m \abs{\frac{x_j}{\prod_{i=j+1}^{j+n} w_i}}^p<\frac{2m \varepsilon^p}{(2m)^p} \leq \varepsilon ^p,
   \]
   which implies
   \[
   \bigcap_{j=-m}^m A_{M,j}\subseteq \{n\in \N: S_ny\in V_\varepsilon\}.
   \]
Consequently, $\{n\in \N: S_ny\in V_\varepsilon\} \in \mathcal{A}_w$, and $B_w$ is an $\F$-operator.

Certainly, condition (3) implies (5). If (5) holds, the argument preceding this Proposition yields that, for each $j\in\N$ and for every $M>0$, the sets $A_{M,j}$ and $\bar{A}_{M,j}$ belong to $\Ft$, which
gives (3) since $\Ft$ is a filter.
\end{proof}

When $\F=\I^*$ is the filter of cofinite sets, we obtain as a consequence the well known characterization of mixing bilateral weighted shifts. On the other hand, the case $\F=\mathcal{S}$ offers again an interesting result.

  \begin{corollary}
  \label{tebws}
  Let $B_w$ be a  bilateral weighted backward shift on $X=c_0(\Z)$ or $\ell^p(\Z)$, $1\leq p< \infty$.  Then the following are equivalent:
\begin{enumerate}
\item[(1)] $B_w$ is a topologically ergodic operator;
\item[(2)] for every $j\in \N$ and $M>0$, $A_{M, j}$ and $\bar{A}_{M, j}$ are syndetic sets.
\end{enumerate}
\end{corollary}


The unilateral version of Proposition~\ref{bilateral.weighted.shifts.F.oper} we provide next relies only on the sets $A_{M,j}$. Notice that for a hypercyclic unilateral weighted shift $B_w$ the collection of sets $\{ A_{M,j}  \ : \ M>0 \mbox{ and } j\in \N\}$ forms a base of a filter (which we call again $\mathcal{A}_w$) since, as before, if  $M_1,M_2>0$ and  $j_1,j_2,j_3\in\N$ with $j_3>\max\{j_1,j_2\}$, then there is $M_3>0$ such that
  \[
  A_{M_3,j_3}\subset A_{M_1,j_1} \cap A_{M_2,j_2}.
  \]
This fact yields a simplification of the corresponding characterization of unilateral weighted shifts that are $\F$-operators, which can be further simplified if $\F$ is a shift$_-$-invariant family.
The unilateral version of Proposition~\ref{bilateral.weighted.shifts.F.oper} can be stated as follows.

\begin{proposition} \label{weighted.shifts.F.oper}
Let $B_w$ be an unilateral weighted backward shift on $c_0(\Z_+)$ or $\ell^p(\Z_+)$ ($1\leq p< \infty$). The following are equivalent:
\begin{enumerate}
\item[(1)] $B_w$ is an $\Ft$-operator;
\item[(2)] $B_w$ is an $\F$-operator;
\item[(3)]  for every $j\in \N$ and $M>0$,  the set $A_{M,j}\in \F$;
\item[(4)] $B_w$ is hypercyclic, $\mathcal{A}_w\subset \F$, and $B_w$ satisfies the $\mathcal{A}_w$-Criterion.
\end{enumerate}
If in addition  $\mathscr{F}$ is shift$_-$-invariant, the above conditions are equivalent to
\begin{enumerate}
\item[(5)] for every $M>0$ the set $A_{M}\in \F$.
\end{enumerate}
\end{proposition}

\begin{proof}
We only prove that if $\mathscr{F}$ is shift$_-$-invariant then condition (5) implies (3). Let $M>0$ and $j\in\N$. We fix $M'>M(\sup_{i\in\N}\abs{w_i})^j$ such that $A_{M'}\subset [j+1,+\infty [$.  Given $n\in A_{M'}$, we have
  \[
  \prod_{s=j+1}^{n}\abs{w_s}=\frac{\prod_{s=1}^{n}\abs{w_s}}{\prod_{s=1}^{j}\abs{w_s}}> \frac{M'}{(\sup_{i\in\N}\abs{w_i})^j}> M.
  \]
This implies that $A_{M'}-j\subset A_{M, j}$.
Since $\mathscr{F}$ is a shift$_-$-invariant family, we conclude that $A_{M, j}\in \F$.
 \end{proof}

In consequence we have the following characterization of topologically ergodic unilateral backward weighted shifts.

\begin{corollary} \label{teuws}
Let $B_w$ be an unilateral weighted backward shift on $X=c_0(\Z_+)$ or $\ell^p(\Z_+)$, $1\leq p< \infty$, then the following are equivalent:
\begin{enumerate}
\item[(1)] $B_w$ is topologically ergodic;
\item[(2)] for every $M>0$ the set $A_{M}$ is syndetic.
\end{enumerate}
\end{corollary}

We conclude this section by considering finite products of $\F$-maps.

\begin{proposition}\label{powandsum}
Let  $T_1,\dots,T_m$ be continuous maps on $X$, then
\begin{enumerate}
\item[(1)] for $n\ge 1$, $T_1^n$ is an $\mathscr{F}$-map on $X$ if and only if $T_1$ is an $\F_n$-map where
$\F_n:=\{A\subset \Z_+: \frac{1}{n}(A\cap n\Z_+)\in \mathscr{F}\}$. In other words, $T_1^n$ is an $\mathscr{F}$-map on $X$ if and only if for every $U,V\in \mathcal{U}(X)$, $N_{T_1}(U,V)\cap n\Z_+ \in n\F$.
\item[(2)] If $\F$ is a filter then $T_1\times T_2\times \dots \times T_m$ is an $\mathscr{F}$-map on $X^m$ if and only if $T_l$ is an $\mathscr{F}$-map on $X$ for every $1\le l\le m$.
\end{enumerate}
\end{proposition}

\begin{proof}
(1) If $n\geq 1$, then $T_1^n$ is an $\mathscr{F}$-map on $X$ if and only if $N_{T^n_1}(U,V)\in \F$ for every $U,V \in \mathcal{U}(X)$. We remark that $N_{T^n_1}(U,V)=\frac{1}{n}(N_{T_1}(U,V)\cap n\Z_+)$. Therefore, $N_{T^n_1}(U,V)\in \F$ if and only if $N_{T_1}(U,V)\in \F_n$.

(2) Note that $T_1\times T_2\times \dots \times T_m$ is an $\mathscr{F}$-map on $X^m$ if and only if $\bigcap_{l=1}^m N_{T_l}(U_l, V_l)\in \F$, for any  $(U_l, V_l)_{l=1}^m\in (\mathcal{U}(X)\times \mathcal{U}(X))^m$. The conclusion follows since $\F$ is a filter.
\end{proof}

Hence by  Proposition \ref{bilateral.weighted.shifts.F.oper} and Proposition~\ref{powandsum} we have the following corollary.
 \begin{corollary}
 \label{corol.weighted.shifts.F.oper}
 Let $\mathscr{F}$ be a filter and $B_w$ be a bilateral weighted backward shift on $X=\ell^p(\Z)$ or $c_0(\Z)$. Then, for every $m\in \N$, the following are equivalent:

(1) $B_w\oplus B_w^2\oplus ... \oplus B_w^m$ is an $\mathscr{F}$-operator on $X^m$;

(2) For every $1\leq l\leq m$, $M>0$ and $j\in \Z$, $A_{M, j}\cap l\Z_+\in l\F$ and $\bar{A}_{M, j}\cap l\Z_+\in l\F$.

 \end{corollary}

 \section{Return sets and densities}\label{return}

  The purpose of this section is to analyze which kind of density properties the sets $N(U, V)$ can have for a given hypercyclic operator, and classify the hypercyclic operators accordingly.
We first recall the definitions of the asymptotic densities and the Banach densities in $\Z_+$.
\begin{definition}
 Let  $A\subseteq \Z_+$ be given. The \emph{upper and lower asymptotic density}  of $A$ are defined respectively by
   \[
   \overline{d}(A)=\limsup_{n\to\infty}\frac{|A\cap \{1, 2, ..., n\}|}{n} \mbox{ and } \underline{d}(A)=\liminf_{n\to\infty}\frac{|A\cap \{1, 2, ..., n\}|}{n}.
   \]
The \emph{upper and lower Banach density} of $A$ are defined by
    \[
    \overline{Bd}(A)=\lim_{s\to\infty}\alpha^s/s  \mbox{\quad and \quad} \underline{Bd}(A)=\lim_{s\to\infty}\alpha_s/s,
    \]
where for each $s\in\Z_+$
 \[\alpha^s=\limsup_{k\to \infty} |A\cap [k+1, k+s]|\quad \text{and} \quad \alpha_s=\liminf_{k\to \infty} |A\cap [k+1, k+s]|.\]
\end{definition}

In general we have $\underline{Bd}(A)\leq \underline{d}(A)\leq \overline{d}(A) \leq \overline{Bd}(A)$ and
  \begin{equation}\label{densitats}
   \underline{d}(A)+ \overline{d}(\Z_+\setminus A) =1 \ \mbox{ and } \
  \underline{Bd}(A)+ \overline{Bd}(\Z_+\setminus A) =1.
  \end{equation}
  We will consider the following families. 
\[\overline{\mathcal{D}}=\{A\subseteq\Z_+: \overline{d}(A)>0\}, \quad  \underline{\mathcal{D}}=\{A\subseteq\Z_+: \underline{d}(A)>0\},\]
\[\underline{\mathcal{BD}}=\{A\subseteq\Z_+: \underline{Bd}(A)>0\}, \quad \overline{\mathcal{BD}}=\{A\subseteq\Z_+: \overline{Bd}(A)>0\},\]
\[\overline{\mathcal{D}}_1=\{A\subseteq\Z_+: \overline{d}(A)=1\}, \quad  \underline{\mathcal{D}}_1=\{A\subseteq\Z_+: \underline{d}(A)=1\},\]
\[\underline{\mathcal{BD}}_1=\{A\subseteq\Z_+: \underline{Bd}(A)=1\}, \quad  \overline{\mathcal{BD}}_1=\{A\subseteq\Z_+: \overline{Bd}(A)=1\}.\]

Notice that each of these families is shift invariant, and that  $\underline{\mathcal{D}}_1$ and $\underline{\mathcal{BD}}_1$ are filters. Moreover,

\begin{enumerate}
\item $\overline{\mathcal{BD}}_1=\mathcal{T}$, the family of thick sets,
\item  $\underline{\mathcal{BD}}=\mathcal{S}$, the family of syndetic sets,
\item  $\overline{\mathcal{BD}} \supset \PS$, the family of piecewise syndetic sets,
\item $\underline{\mathcal{BD}}_1 \subset \mathcal{TS}$, the family of thickly syndetic sets,
\item $\underline{\mathcal{BD}}_1=\overline{\mathcal{BD}}^*$, $\underline{\mathcal{D}}_1=\overline{\mathcal{D}}^*$, $\overline{\mathcal{D}}_1= \underline{\mathcal{D}}^*$, and $\overline{\mathcal{BD}}_1=\underline{\mathcal{BD}}^*$ by \eqref{densitats}.
\end{enumerate}

In consequence,  $T$ is weakly mixing if and only if $T$ is a  $\overline{\mathcal{BD}}_1$-map.

Weighted shift operators and Proposition~\ref{weighted.shifts.F.oper} help us to provide some counterexamples which allow us to distinguish the different notions of $\F$-operators.

\begin{proposition}
\label{yellow}
 Let $X=c_0(\Z_+)$, then

(1) there exists a $\overline{\mathcal{BD}}_1$-operator which is not $\overline{\mathcal{D}}$-operator.

(2) there exists a $\overline{\mathcal{D}}_1$-operator which is not $\underline{\mathcal{D}}$-operator.

(3) there exists a $\underline{\mathcal{D}}_1$-operator which is not $\underline{\mathcal{BD}}$-operator.

\end{proposition}
\begin{proof}
(1) Consider the weight sequence
\[
w=(\underbrace{1, \dots, 1}_{m_0}, 2,  2^{-1}, \underbrace{1, \dots, 1}_{m_1}, 2, 2, 2^{-2}, \underbrace{1, \dots, 1}_{m_2}, 2, 2, 2, 2^{-3}, \underbrace{1, \dots, 1}_{m_3}, \dots)
\]
We first observe that $\sup_n\prod_{i=1}^nw_i$ is infinite, hence $B_w$ is weakly mixing, see Chapter 4 in \cite{GrPe}. In other words $B_w$ is $\overline{\mathcal{BD}}_1$-operator.

On the other hand, by Proposition \ref{weighted.shifts.F.oper}, we know that it suffices to show that
$\overline{d}(A_1)=0$ in order to deduce that $B_w$ is not a $\overline{\mathcal{D}}$-operator. In other words, it suffices to show that $\overline{d}\Big(\Big\{n\in \N: \prod_{i=1}^n w_i> 1\Big\}\Big)=0$ and this holds if $(m_k)$ grows sufficiently rapidly.

(2) Consider the weight
\[
w=\big(\underbrace{1,\cdots, 1}_{m_0}, \underbrace{2,\cdots, 2}_{n_0}, \underbrace{2^{-n_0}, 1,\cdots, 1}_{m_1}, \underbrace{2,\cdots, 2}_{n_1}, \underbrace{2^{-n_1}, 1,\cdots, 1}_{m_2}, \underbrace{2,\cdots, 2}_{n_2}, \cdots\big).
\]
Thanks to Proposition \ref{weighted.shifts.F.oper}, it suffices to find sequences $(m_k)_k, (n_k)_k$ such that
 \begin{itemize}
 \item $\overline{d}\big(\{n: \prod_{i=1}^n w_i=1\}\big)=1$
 \item $\overline{d}(A_M)=\overline{d} \big(\{n: \prod_{i=1}^n w_i>M \}\big)=1$, for every $M>0$.
 \end{itemize}
Indeed, if $\overline{d}\big(\{n: \prod_{i=1}^n w_i=1\}\big)=1$ then
\[\underline{d}\big(\{n: \prod_{i=1}^n w_i>1\}\big)=1-\overline{d}\big(\{n: \prod_{i=1}^n w_i\le 1\}\big)=0.\]
Define sequences of intervals in the following way: $ \A_{k}=[10^{2^{2k+1}}, 10^{2^{2k+2}}[$ and $\B_k=[10^{2^{2k+2}}, 10^{2^{2k+3}}[$ for every $k\in \Z_+$.

So $\A=\bigcup_{k\in \N}\A_k$ and $\B=\bigcup_{k\in \N}\B_k$ are disjoint with $\overline{d}(\A)=\overline{d}(\B)=1$. Hence, setting $m_k=|\A_k|, n_k=|\B_k|$ for every $k$, we are done.

(3) Let $m_k=10^{2^k}$ for every $k\in \Z_+$. We consider the weight
\[
w=\big(1, 2, 2^{-1}, 1, 1, \underbrace{2, \cdots, 2}_{m_0}, 2^{-m_0}, 1, 1, 1, \underbrace{2, \cdots, 2}_{m_1}, 2^{-m_1}, 1, 1, 1, 1, \]
\[\underbrace{2, \cdots, 2}_{m_2}, 2^{-m_2}, \cdots\big).
\]
The set $A_1=\{n: \prod_{i=1}^n w_i>1\}$ has arbitrarily large gaps, hence $B_w$ is not an $\underline{\mathcal{BD}}$-operator by Proposition \ref{weighted.shifts.F.oper}. On the other hand, we have for every $M>1$
\[
\underline{d}(A_M)=\underline{d}\big(\{n: \prod_{i=1}^n w_i>M \}\big)=1.
\]
Hence, $B_w$ is $\underline{\mathcal{D}}_1$-operator by Proposition \ref{weighted.shifts.F.oper}.
\end{proof}

  Mixing operators obviously are $\underline{\mathcal{BD}}_1$-operators, but the converse is false, this is the argument of the next result.

 \begin{proposition}
 There exists a $\underline{\mathcal{BD}}_1$-operator on $c_0(\Z_+)$ which is not mixing.
 \end{proposition}

 \begin{proof}
 Consider the weight $w=(w_n)_{n=1}^\infty$ defined by
 \[
  w=(1, 2,2^{-1},2,2,2^{-2}, \dots ,\underbrace{2, \dots, 2}_{n}, 2^{-n}, \dots).
 \]

The weighted shift $B_w$ is not mixing since $\prod_{i=1}^nw_i$ does not tend to infinity as $n$ tends to infinity (see, e.g., Chapter 4 in \cite{GrPe}). It remains to show that $\underline{Bd}(A_M)=1$ for every $M\ge 1$. Let $M> 1$ and $n\in \N$  such that $2^{n-1}<M\leq 2^{n}$. If $k>n(n+1)/2$ and
 $s\geq (n+1)+(n+2)+\dots +2n=n(3n+1)/2$, then there is $l_s>1$ such that $(l_s-1)n((l_s+1)n+1)/2\leq s<(l_s)n((l_s+2)n+1)/2$. An easy computation shows that we have $|A_M\cap [k, k+s]|\geq s- l_s(n^2+n)>(l_s^2/2-l_s-1)n^2-l_sn$. Therefore,
\[
\alpha_s:=\liminf_{k\to \infty}|A_M\cap [k, k+s]|\geq (l_s^2/2-l_s-1)n^2-l_sn,
\]
and thus
\[
\underline{Bd}(A_M)=\lim_{s\to \infty}\frac{\alpha_s}{s}\geq \lim_{s\to \infty}\frac{(l_s^2/2-l_s-1)n^2-l_sn}{(l_s^2/2+l_s)n^2+l_sn}=1.
\]
We conclude by Proposition~\ref{weighted.shifts.F.oper}.
\end{proof}

 \begin{proposition}
 \label{green}
There exists a $\underline{\mathcal{BD}}$-operator on $\ell^1(\Z_+)$ which is not a $\overline{\mathcal{D}}_1$-operator.
 \end{proposition}
 \begin{proof}
  Let $\A_n=[\underbrace{2, \dots, 2}_{n-times}, 2^{-n}]$, $\B_1=\A_1$, $\B_n=[\B_{n-1}, \A_n, \B_{n-1}]$, and consider the weight sequence
  \[
  w=(\underbrace{\underbrace{\A_1, \A_2, \A_1}, \A_3, \underbrace{\A_1, \A_2, \A_1}}, \A_4, \underbrace{\underbrace{\A_1, \A_2, \A_1}, \A_3, \underbrace{\A_1, \A_2, \A_1}}, \dots).
  \]
 Since $A_M$ has bounded gaps for every $M>0$, we have from Corollary~\ref{teuws} that $B_w$ is topologically ergodic, i.e., it is a $\underline{\mathcal{BD}}$-operator.

In view of Proposition~\ref{weighted.shifts.F.oper}, it now suffices to show that
   \[
   \overline{d}\Big(\{k\in \N:\prod_{i=1}^k |w_i|>1\}\Big)<1.
   \]
We first notice  that
\[
|\B_n|=3\cdot2^n-n-3 \ \mbox{ and } \ \beta_n:=\Big|\big\{k\leq |\B_n|: \prod_{i=1}^k |w_i|=1\big\}\Big|=2^n-1.
\]
Now we observe that $\prod_{i=1}^k |w_i| \geq 1$ for all $k\in\N$.
Therefore, we have
\begin{align*}
\overline{d}\Big(\Big\{k\in \N: \prod_{i=1}^k |w_i|>1\Big\}\Big)&=\limsup_n\frac{\Big|\big\{k\in [1, n]: \prod_{i=1}^k |w_i|>1\big\}\Big|}{n}\\
&=\limsup_n\frac{\Big|\big\{k\leq |\B_n|+n+1: \prod_{i=1}^k |w_i|>1\big\}\Big|}{|\B_n|+n+1}\\
&=\lim_n \frac{|\B_n|-\beta_n+n+2}{|\B_n|+n+1}=\lim_n \frac{2 \cdot 2^n }{3\cdot 2^n-2}=\frac{2}{3}<1.
\end{align*}
 \end{proof}

 Figure \ref{fig:fig1} below summarizes the results of this section. We remark that:
 
  \begin{itemize}
  \item by Proposition \ref{yellow} (1), there exists a $\overline{\mathcal{BD}}_1$-operator which is not a $\overline{\mathcal{D}}_1$-operator and a $\overline{\mathcal{BD}}$-operator which is not a $ \overline{\mathcal{D}}$-operator;
  \item by Proposition \ref{yellow} (2), there exists a $\overline{\mathcal{D}}_1$-operator which is not a $\underline{\mathcal{D}}_1$-operator and a $\overline{\mathcal{D}}$-operator which is not a $ \underline{\mathcal{D}}$-operator;
  \item by Proposition \ref{yellow} (3), there exists a $\underline{\mathcal{D}}_1$-operator which is not a $\underline{\mathcal{BD}}_1$ and a $\underline{\mathcal{D}}$-operator which is not a $ \underline{\mathcal{BD}}$-operator.

  On the other hand, by Proposition \ref{green}, there exists a
 \item  $\underline{\mathcal{BD}}$-operator which is not a $\underline{\mathcal{BD}}_1$-operator;
 \item $\underline{\mathcal{BD}}$-operator which is not a $\overline{\mathcal{D}}_1$-operator;
 \item  $\underline{\mathcal{D}}$-operator which is not a $\underline{\mathcal{D}}_1$-operator;
 \item  $\underline{\mathcal{D}}$-operator which is not a $\overline{\mathcal{D}}_1$-operator;
 \item  $\overline{\mathcal{D}}$-operator which is not a $\overline{\mathcal{D}}_1$-operator.

  \end{itemize}

\begin{figure}[H]
 \centering
 \includegraphics[width=0.9\textwidth]{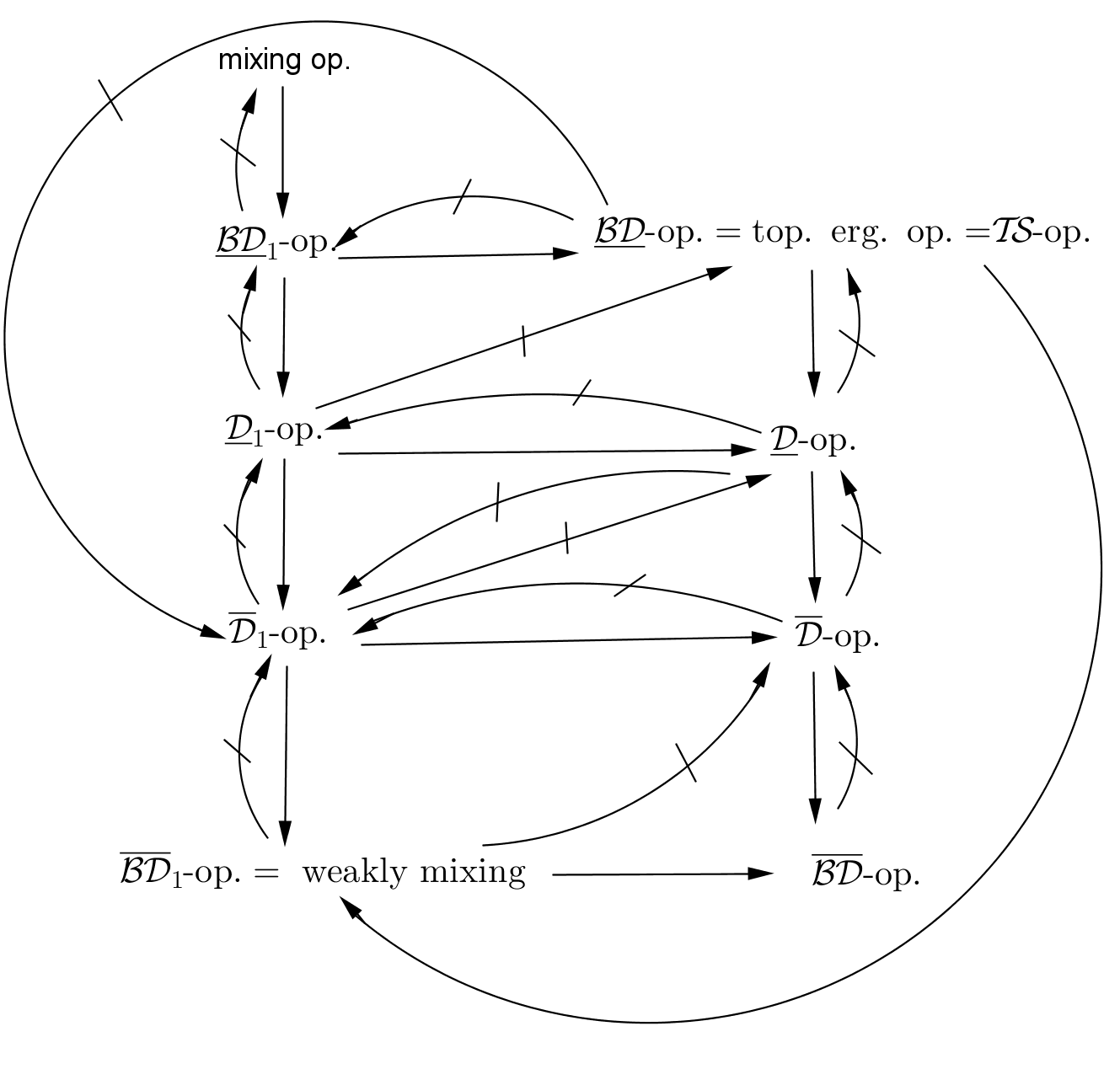}  
 \caption{Densities and transitivity properties}
 \label{fig:fig1}
\end{figure}

\section{Some special families}
\label{filt}

 In this section we study new classes of $\F$-transitive operators given by families commonly used in Ramsey Theory. For a rich source on these families see \cite{HiSt}. For instance, we will consider the families of $\Delta$-sets and of $\IP$-sets, as well as their dual families.

     \[
\begin{aligned}
\Delta&=\{A\subseteq \Z_+: (B-B)\cap \Z_+ \subseteq A, \mbox{ for some infinite subset } B \mbox { of } \Z_+\} \\
\IP &=\{A\subseteq \Z_+: \exists (x_n)_n\subseteq \N \mbox{ with } \sum_{n\in F}x_n\in A, \ \forall F\subset  \Z_+ \mbox{ finite}\}.
\end{aligned}
\]
The families $\Delta^*$ and $\IP^*$ are filters since  $\Delta$ and $\IP$  are partition regular.
In addition, we have
\[
  \I^*\subsetneqq \Delta^* \subsetneqq \IP^* \subsetneqq \Synd
   \]

   \begin{equation}
  \label{hierarchy.sets}
   \I^*\subsetneqq \PS^* \subsetneqq \Synd,
  \end{equation}
  see \cite{BeDo} for details.
In linear dynamics, some of the widely studied classes are the mixing and weakly mixing operators.
 As we already mentioned, an operator $T$ is mixing if and only if it is an $\I^*$-operator and $T$ is weakly mixing if and only if $T$ is  a  $\mathcal{T}$-operator.
We recall that the class  of $\mathcal{TS}$-operators coincides with the class  of topologically ergodic operators by Lemma~\ref{lthick} (see also the exercises  in  \cite[Chapter 2]{GrPe}). Moreover, since $\mathcal{TS}=\mathcal{PS}^*$ and $\mathcal{TS}$ is a filter, we know that $\mathcal{PS}^*$ is partition regular (Lemma~\ref{part.reg}). With the help of Proposition~\ref{Fstar.maps} applied to $\F=\PS$ we can therefore complete the picture.

\begin{proposition} \label{synd.op.is.hered.synd}
 Let $T\in\Lin(X)$, where $X$ is a separable $F$-space. The following are equivalent:
\begin{enumerate}
\item[(1)] $T$ is a topologically ergodic operator;
\item[(2)] $T$ is a hereditarily $\mathcal{TS}$-operator;
\item[(3)] $T$ is a $\mathcal{TS}$-operator;
\item[(4)] $T$ is a hereditarily $\PS$-operator;
\item[(5)]  $hcA:=\{ x\in X : \overline{\{T^nx:n\in A\}}=X\}$ is a dense ($G_\delta$) set in $X$ for any $A\in \PS$.
\end{enumerate}
\end{proposition}

We will distinguish different classes of $\F$-operators by means of Proposition~\ref{weighted.shifts.F.oper}.
Given a family $\F$, the following are two standard ways to induce  shift-invariant families
  \[
  \F_+:=\bigcup_{k\in\Z}(\F+k)
  \]
  \[
  \F_\bullet:=\bigcap_{k\in\Z} (\F+k),
  \]
where $\F+k:=\{ A\subset \Z_+ : \exists B\in \F \mbox{ with } (B+k)\cap \Z_+\subset A\}$, $k\in\Z$. We have
\[
\Ft\subset \F_\bullet \subset \F \subset \F_+ .
\]
Moreover, for any $A\subseteq \Z_+$ we have
  \begin{equation}
  A\in\left(\F^*\right)_\bullet \mbox{ if and only if } A\in \left(\F_+\right)^*.
  \label{equiv.Fint}
  \end{equation}
It is  well-known that $\Delta$ and $\IP$ are not shift invariant, while $\PS$  is. Also, if
$\F=\Delta, \IP$ or $\PS$ and $\mathscr{G}=\F$ or $\F_+$ then
 \begin{equation}
 \label{equiv.F^*}
 A\in \mathscr{G}^* \mbox{ if and only if }\Z_+\setminus A\notin \mathscr{G},
 \end{equation}
since $\mathscr{G}$ is partition regular.

\begin{proposition}\label{fbullet}
 Every $\mathscr{F}$-operator is  an $\mathscr{F}_{\bullet}$-operator.
\end{proposition}

\begin{proof}
Let $U,V\in \mathcal{U}(X)$ and $k\ge 0$. We have $N(U,T^{-k}V)+k\subset N(U,V)$.  Moreover, since $X$ has no isolated points, by transitivity we can find non-empty open sets $U'\subset U$ and $V'\subset V$ such that $N(T^{-k}U',V')\subset [k,+\infty [$. Thus we have
\[
 N(T^{-k}U',V')\big)-k\subset N(U,V).
 \]
 We can conclude that every $\mathscr{F}$-operator is an  $\mathscr{F}_{\bullet}$-operator.
\end{proof}

We  next compare the notions of mixing operator, $\Delta^*$-operator, $\IP^*$-operator and topologically ergodic operator.

  \begin{proposition}
  \label{proposition.first}
There exists a topologically ergodic weighted backward shift on $X=c_0(\Z_+)$ or $\ell^p(\Z_+)$, $1\leq p< \infty$, which is not an $\IP^*$-operator.
\end{proposition}
\begin{proof}

  Consider the set
   \[
   B=\Big\{\sum_{n\in F} 2^{2n}: F \mbox{ finite set of } \N\Big\}.
   \]
   Clearly $B\in \IP$ and thus $\Z_+\setminus B\notin \IP^*$ by (\ref{equiv.F^*}). Let $(b_n)$ be the increasing enumeration of $B$. We define the weight $w=(w_m)_{m=1}^\infty$ as follows

    \begin{equation}
      \label{formula}
  w=(2,\dots, 2, \underbrace{\frac{1}{2^{b_1-1}}}_{w_{b_1}}, 2, \dots, 2, \underbrace{\frac{1}{2^{b_2-b_1-1}}}_{w_{b_2}}, 2, \dots, 2, \underbrace{\frac{1}{2^{b_3-b_2-1}}}_{w_{b_3}}, 2, \dots ).
  \end{equation}
  Now, $A_1:=\{n\geq 1: \prod_{i=1}^n w_i>1\}=\Z_+\setminus B$, hence $B_w$ is not an $\IP^*$-operator by Propositon~\ref{weighted.shifts.F.oper}. On the other hand, it is easy to see that $B\notin \PS$. Then $(B+i)\notin \PS$ for every $i\geq 0$, since $\PS$ is shift invariant. Hence, by (\ref{equiv.F^*}) the set $\Z_+\setminus (B+i)\in \PS^*$ for every $i\geq 0$. Now observe that $A_{2^j}:=\{n\geq 1: \prod_{i=1}^n w_i>2^j\}=\Z_+\setminus \Big(\bigcup_{i=0}^jB+i\Big)=\bigcap_{i=0}^j(\Z_+\setminus (B+i))\in \PS^*$, since $\PS^*$ is a filter. Hence $B_w$ is a $\PS^*$-operator, or equivalently a topologically ergodic operator, by Proposition~\ref{weighted.shifts.F.oper}.
   \end{proof}


\begin{proposition}
\label{proposition.second}
There exists a weighted backward shift operator on $X=c_0(\Z_+)$ or $\ell^p(\Z_+)$, $1\leq p< \infty$, which is an $\IP^*$-operator but not a $\Delta^*$-operator.
\end{proposition}

\begin{proof}
Let $B$ be an infinite subset of $\N$ with unbounded gaps and let $(b_n)_n$ be an increasing enumeration of $B$. So there exists an increasing sequence $(n_k)$ such that
 \begin{equation}
 \label{cond.inicial}
 b_{n_k+1}-b_{n_k}\to \infty.
 \end{equation}
  Consider the weight sequence $w=(w_m)_{m=1}^\infty$ given by (\ref{formula}). As before $\{n \geq 1: \prod_{i=1}^n w_i>1\}=\Z_+\setminus B$, so it would be desirable that $B\in \Delta$ and thus that $\Z_+\setminus B\notin \Delta^*$ since this would imply that $B_w$ is not a $\Delta^*$-operator.

On the other hand, it can be verified that for every $M>0$ and $j\in \N$ there exists a finite subset $F$ of $\Z$ such that $A_{M, j}=\Z_+\setminus (\cup_{i\in F}B+i)$. Hence, in order to conclude that $B_w$ is an $\IP^*$-operator, by Proposition~\ref{weighted.shifts.F.oper} and condition (\ref{equiv.F^*}) we need to verify

  \begin{equation}
  \label{eq.union.IP}
  \bigcup_{i\in F} (B+i)\notin \IP
  \end{equation}
  for any finite subset $F$ of $\Z$. Now, since $\IP$ is partition regular, condition (\ref{eq.union.IP}) is obtained if $B\notin \IP_+$  and this in turn is equivalent to $\Z_+\setminus B\in \big(\IP^*\big)_\bullet$ by (\ref{equiv.F^*}) and (\ref{equiv.Fint}). Now, an obvious modification in the proof of \cite[Theorem 2.11
(1)]{BeDo} ensures the existence of a set $E\in \big(\IP^*\big)_\bullet$ which is not $\bigcup_{n\in \Z_+}\big(\Delta^*+n\big)$-set in $\N$, hence not $\Delta^*$-set. In addition, $\Z_+\setminus E$ has unbounded gaps. Setting $B=\Z_+\setminus E$ we are done.
 \end{proof}

Evidently, every mixing operator is a $\Delta^*$-operator but the converse is not true.

\begin{proposition}
\label{proposition.quasimixing}
There exists a $ \Delta^* $-weighted backward shift on $c_0(\Z_+)$ or $\ell^p(\Z_+)$, $1\leq p< \infty$, which is not mixing.
 \end{proposition}

 \begin{proof}

Let $B=\{b_i: b_1=2, b_{i+1}=b_i+i+2, i\in \N\}$. Consider the weight sequence $w=(w_m)_{m=1}^\infty$ given by (\ref{formula}), so we have

  \[
  w=\big(2, 2^{-1}, 2, 2, 2^{-2}, 2, 2, 2, 2^{-3}, \dots \big).
  \]
  We know that $B_w$ is not mixing since $\prod_{i=1}^n w_i$ does not tend to infinity as $n$ tends to infinity. On the other hand, it can be verified that for every $M>0$ and $j\in \N$ there exists a finite subset $F$ of $\Z$ such that $A_{M, j}=\Z_+\setminus (\cup_{i\in F}B+i)$. Hence, in order to conclude that $B_w$ is a $\Delta^*$-operator, by Proposition~\ref{weighted.shifts.F.oper} and condition (\ref{equiv.F^*}) we need to verify $\bigcup_{i\in F} B+i\notin \Delta$, for every finite subset $F$ of $\Z$.

So, let $F$ be a finite subset of $\Z$ with $N=\max_{a,b\in F}\abs{a-b}$. Suppose that $\bigcup_{i\in F} B+i$ is a $\Delta$-set. Then, there exists an increasing sequence $(d_m)_m$ such that $\bigcup_{i\in F} B+i=\Delta\big((d_m)_m\big)$, where $\Delta\big((d_m)_m\big)$ denotes the \emph{set of differences of $(d_m)_m$} defined  by $\Delta\big((d_m)_m\big)=\{d_j-d_i: 1\leq i<j\}$. Fix $d_{j_1}, d_{j_2}(j_1<j_2)$ such that $\abs{d_{j_2}-d_{j_1}}>N$. Then for each $m\in\N$ we have
   \[
   \abs{d_{j_2}-d_{j_1}}=\abs{(d_{j_m}-d_{j_1})-(d_{j_m}-d_{j_2})},  
   \]
   which means that the distance $\abs{d_{j_2}-d_{j_1}}$ between elements of $\bigcup_{i\in F} B+i$ is attained infinitely many times, which is not the case taking into account the way in which $B$ was defined. We thus conclude that $\bigcup_{i\in F} B+i\notin \Delta$.
   \end{proof}


\subsection{Connection with $\A$-hypercyclicity}\label{ahyp}

 In this subsection we investigate the connection between the classes of hypercyclic operators considered throughout this work and the notion of $\A$-hypercyclicity studied in \cite{BMPP}.

Given a family $\A$ on $\Z_+$, an operator  $T\in \Lin(X)$ is called \emph{$\A$-hypercyclic} if there exists $x\in X$ such that  $N(x, V)\in \A$ for each $V$ in $\mathcal{U}(X)$. Such a vector $x$ is called an \emph{$\A$-hypercyclic vector} for $T$.

 When $\A=\underline{\mathcal{D}}$, the operator $T$ is called \emph{frequently hypercyclic}. This class was introduced by Bayart and Grivaux in \cite{BG1}, \cite{BG}.
When $\A=\overline{\mathcal{D}}$, the operator $T$ is called \emph{$\mathfrak{U}$-frequently hypercyclic}; this class was introduced by Shkarin \cite{Sh1}.
When $\A=\overline{\mathcal{BD}}$, the operator $T$ is called \emph{reiteratively hypercyclic} \cite{P05} (see a detailed study in \cite{BMPP}).

 The frequently hypercyclic operators constitute by far the most extensively studied class of operators amongst the three classes mentioned above. Clearly any frequently hypercyclic operator is an $\mathfrak{U}$-frequently hypercyclic operator, which in turn is reiteratively hypercyclic.
The hierarchy between frequently hypercyclic and $\mathfrak{U}$-frequently hypercyclic operators as well as a full characterization for weighted shift operators have been established by Bayart and Ruzsa \cite{BaRu}. A complementary study of this kind, taking into account reiterative hypercyclicity can be found in~\cite{BMPP}.

In particular, we already know that there exists a mixing weighted shift which is not reiteratively hypercyclic as shown in \cite{BMPP}. On the other hand, there exists a frequently hypercyclic (hence reiteratively hypercyclic) operator which is not mixing, see \cite{BaGr}. Reiteratively hypercyclic operators are topologically ergodic \cite{BMPP,GrPe05}. One can therefore wonder whether any reiteratively hypercyclic operator is a $\Delta^*$-operator or an $\IP^*$-operator.

  \begin{proposition}
  \label{propP.implies.shift.DeltaStar}
     Let $T\in\mathcal{L}(X)$ be a reiteratively hypercyclic operator. Then
   \[
  N(U, V)\in \bigcap_{t\in N(U, V)}\Big(\Delta^*+t\Big),
\]
for every $U, V$ non-empty open sets in $X$.
\end{proposition}
\begin{proof}
 Let $U, V\in \mathcal{U}(X)$ and $n\in N(U, V)$. The set $U_n=U\cap T^{-n}V$ is a non-empty open set. Since $T$ is reiteratively hypercyclic, there exists $x\in X$ such that $\overline{Bd}\left(N(x, U_n)\right)>0$.

Let $s_1, s_2\in N(x, U_n)$. We have
\[
 T^{s_1-s_2+n}(T^{s_2}x)=T^{n}(T^{s_1}x)\in V.
\]
In other words,
 \begin{equation}
 N(x, U_n)-N(x, U_n)+n\subseteq N(U, V).
  \label{ident.grivaux1}
 \end{equation}
The desired result then follows from Theorem~3.18 in~\cite{Furstenberg}, which implies that $A-A\in \Delta^*$ whenever $A\in \overline{\mathcal{BD}}$.
\end{proof}

The family $\Delta^*$ is not shift invariant ($2\N:=\{2n: n\in \N\}\in \Delta^*$ while $2\N+1\notin\Delta^*$). Hence, we cannot deduce from Proposition~\ref{propP.implies.shift.DeltaStar} that every reiteratively hypercyclic operator is a  $\Delta^*$-operator.
In fact, we are not able to answer in general the following question: is any reiteratively hypercyclic operator either a  $\Delta^*$-operator or an $\IP^*$-operator? However we can show that the answer is yes if we consider bilateral or unilateral weighted shifts.

\begin{proposition}
\label{bilateral.rhc.implies.delta^*}
If $B_w$ is reiteratively hypercyclic on $X=\ell^p(\Z)$, $1\leq p <\infty$, or $X=c_0(\Z)$, then $B_w$ is an $\Delta^*$-operator.
\end{proposition}
In order to prove Proposition~\ref{bilateral.rhc.implies.delta^*}, we first state two lemmas. The first one directly follows from Proposition~\ref{propP.implies.shift.DeltaStar}. 
\begin{lemma}
\label{UintersectsV}
Let $U, V$ non-empty open sets in $X$ such that $U\cap V\neq\emptyset$, if $T$ is reiteratively hypercyclic on $X$ then $N(U, V)\in \Delta^*$.
\end{lemma}

Let $X=\ell^p(\Z)$, $1\leq p <\infty$, or $c_0(\Z)$. The second lemma will rely on the non-empty open sets $U_{R,j}$ defined for every $R>1$ and every $j\in \Z$ by
 \[
 U_{R, j}=\{U\in \mathcal{U}(X): \abs{x_j}>\frac{1}{R}, \forall x\in U\}.
 \]
 In particular, we remark that if $MR>1$ then $B((M+1)e_j; \frac{1}{MR})\in U_{R, j}$, where $B(y; \epsilon)$ stands for the open ball centered at $y$ with radius $\epsilon$.
\begin{lemma}
\label{cond.on.open}
Let $M>0$, $j\in \Z$ and $R>1$ such that  $MR>1$. Suppose there exists $U\in U_{R, j}$ such that for any non-empty open subset $\widetilde{U}$ of $U$ it holds  $N(\widetilde{U}, B((M+1)e_j; \frac{1}{MR}))\in \Delta^*$. Then $A_{M,j}\in \Delta^*$ and $\bar{A}_{M,j}\in \Delta^*$.
\end{lemma}
\begin{proof}
Let $(z(m))_m$ be a dense set in $X$ such that \[z(m)=(z(m)_1, \dots, z(m)_m, 0\dots)\] and $U_m=B(z(m); 1/m)$. Let $U\in U_{R, j}$ such that for any non-empty open subset $\widetilde{U}$ of $U$ we have  $N(\widetilde{U}, B((M+1)e_j; \frac{1}{MR}))\in \Delta^*$. Then there exists $m$ such that $U_m\subset U$ and hence $N(U_m, B((M+1)e_j; \frac{1}{MR}))\in \Delta^*$. Pick $r\in N(U_m, B((M+1)e_j; \frac{1}{MR}))$ with $r>m$ and $x\in U_m$ such that $B_w^rx\in B((M+1)e_j; \frac{1}{MR})$.

 Then, we have
 \begin{equation}
 \label{equation.for.j}
 \abs{\left(\prod_{i=j+1}^{j+r} w_i\right) x_{j+r}-(M+1)}<\frac{1}{MR}
 \end{equation}
and for every $t\neq j$
  \begin{equation}
 \label{equation.for.t.different.from.j}
\abs{\left(\prod_{i=t+1}^{t+r} w_i\right) x_{t+r}}<\frac{1}{MR}.
 \end{equation}
By \eqref{equation.for.j} we get,
 \[
 \abs{\prod_{i=1}^r w_{i+j}}>\abs{\prod_{i=1}^r w_{i+j}x_{r+j}}>  M,
 \]
 where the first inequality follows since $r>m$. We conclude that $N(U_m, B((M+1)e_j; \frac{1}{MR}))\setminus \{1\dots m\}\subseteq A_{M,j}$ and thus $A_{M,j}\in \Delta^*$.

  On the other hand, by \eqref{equation.for.t.different.from.j}, we get $\prod_{i=j-r+1}^j \abs{w_i x_j}<\frac{1}{MR}$, hence
  \[
  \prod_{i=j-r+1}^j \abs{w_i}\frac{1}{R}< \prod_{i=j-r+1}^j \abs{w_i x_j}<\frac{1}{MR}.
  \]
We deduce that $\prod_{i=j-r+1}^j \abs{w_i}<\frac{1}{M}$ and thus $\bar{A}_{M,j}\in \Delta^*$.

\end{proof}

\textit{Proof of Proposition \ref{bilateral.rhc.implies.delta^*}}

Suppose $B_w$ is not a $\Delta^*$-operator on $X$, then by Proposition \ref{bilateral.weighted.shifts.F.oper}, there exists $M>0$ and $j\in \Z$ such that $A_{M,j}\notin \Delta^*$ or $\bar{A}_{M,j}\notin \Delta^*$. Let $R>1$ such that $MR>1$. By Lemma \ref{cond.on.open}, it follows that
\[
\forall U\in U_{R, j} \quad \exists \widetilde{U}\subseteq U: N(\widetilde{U}, B((M+1)e_j; \frac{1}{MR}))\notin\Delta^*.
\]
Since $ B((M+1)e_j; \frac{1}{MR})\in U_{R, j}$, we can consider $U=B((M+1)e_j; \frac{1}{MR})$ and there thus exists $\widetilde{U}\subseteq U$ such that $N(\widetilde{U}, U)\notin\Delta^*$, which by Lemma~\ref{UintersectsV}, is not possible if $B_w$ is reiteratively hypercyclic. This concludes the proof of Proposition \ref{bilateral.rhc.implies.delta^*}.

Analogously, we have the following result for unilateral weighted shifts.
\begin{proposition}
\label{rhc.implies.delta^*}
If $B_w$ is reiteratively hypercyclic on $X=\ell^p(\Z_+)$, $1\leq p <\infty$, or on $X=c_0(\Z_+)$, then $B_w$ is a  $\Delta^*$-operator.
\end{proposition}

\begin{proposition}\label{rhc_not_udens1}
There exists a reiteratively hypercyclic operator on $c_0(\mathbb{Z}_+)$ which is not a  $\overline{\mathcal{D}}_1$-operator.
\end{proposition}

\begin{proof}
Let $B_w$ be the weighted shift on $c_0(\mathbb{Z}_+)$  given by
\[w_k=\left\{
\begin{array}{cl} \displaystyle{2} & \quad\mbox{if $k\in S$}\\
 {\displaystyle \prod_{\nu=1}^{k-1}w_{\nu}^{-1}} & \quad\mbox{if $k\in (S+1)\backslash S$}\\
  1 & \quad\mbox{otherwise}.
\end{array}\right.\]
where $S:=\bigcup_{j,l\ge 1}]l 10^j-j,l 10^j+j[$.
It is shown in \cite[Theorem 17]{BMPP} that $B_w$ is reiteratively hypercyclic and that
\[
\overline{d}(\{k\in \mathbb{N}:\prod_{i=1}^k|w_i|\ge 2^j\})\to 0.
\]
In particular, we deduce that there exists $j\ge 1$ such that $\overline{d}(\{k\in \mathbb{N}:\prod_{i=1}^k|w_i|\ge 2^j\})<1$ and in view of Proposition~\ref{weighted.shifts.F.oper}, we can conclude that $B_w$ is not a $\overline{\mathcal{D}}_1$-operator.
\end{proof}

Figure \ref{fig:fig2} summarizes what we know after this work.

\begin{figure}[H]
 \centering
 \includegraphics[width=0.8\textwidth]{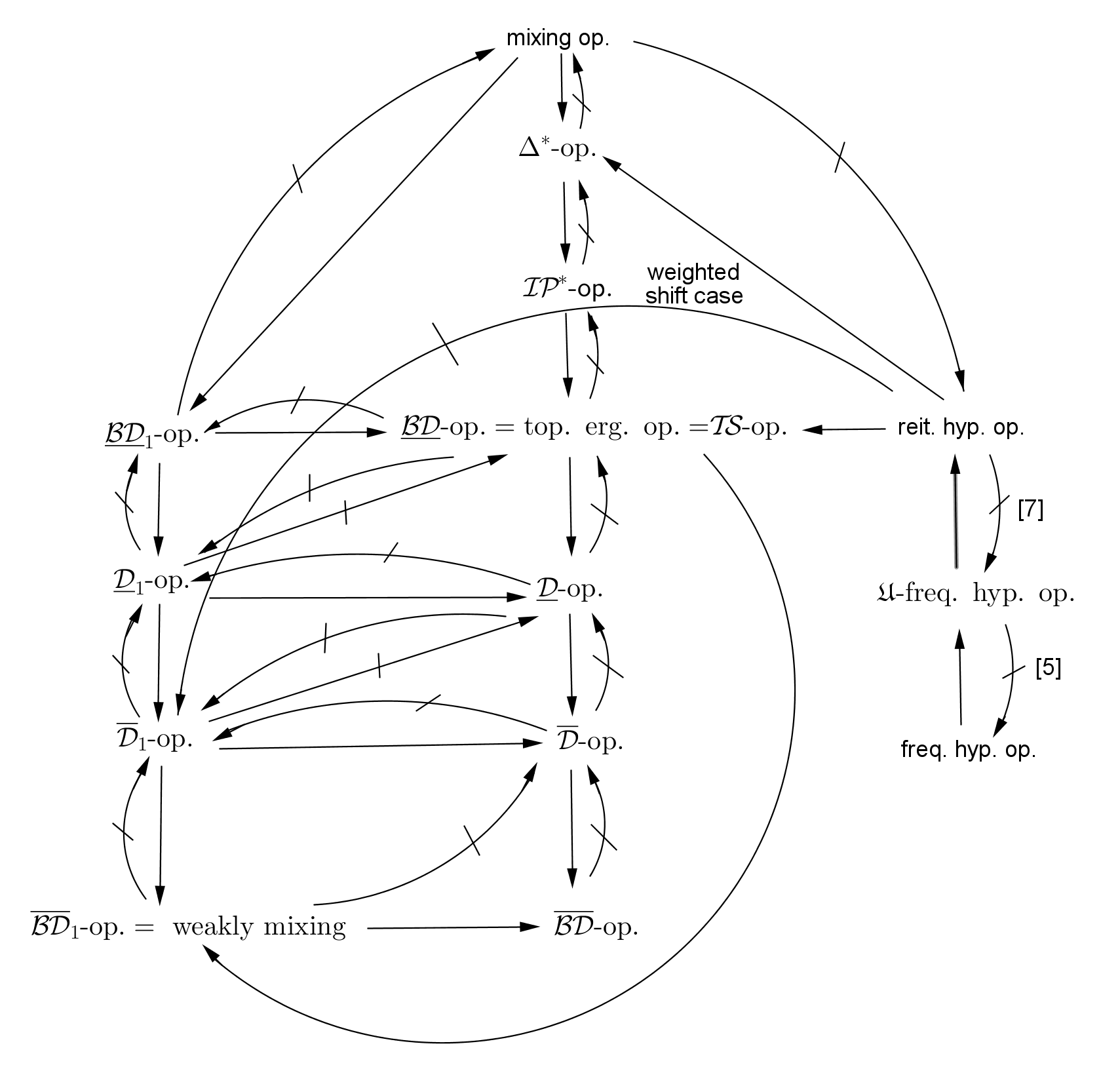}      
 \caption{Known relations}
 \label{fig:fig2}
\end{figure}

 We recall the following questions that remain open.
  \begin{question}
  Does there exist a $\underline{\mathcal{D}}$-operator which is not a $\overline{\mathcal{BD}_1}$-operator? In other words, does there exist $T\in \Lin(X)$ being a $\underline{\mathcal{D}}$-operator but not weakly mixing?
  \end{question}

  Note that if it were the case, then such operator $T$ must not be weighted shift.
\begin{question}
Is any reiteratively hypercyclic operator an $\Delta^*$-operator or an $\IP^*$-operator?
\end{question}


\end{document}